\documentclass[10.5pt]{amsart}
\usepackage{ytableau}
\usepackage{amsthm}
\usepackage{amssymb}
\usepackage{latexsym}
\usepackage{multicol}
\usepackage{etoolbox} 
\usepackage{verbatim,enumerate}
\usepackage{accents}
\usepackage[utf8]{inputenc}
\usepackage{hyperref}
\usepackage{amsmath, amscd}
\usepackage{soul}
\usepackage{tikz} 
\usepackage{mathtools}
\usetikzlibrary{matrix,arrows,decorations.pathmorphing,arrows.meta,calc}
\newlength\cellsize 

\savebox2{%
\begin{picture}(10,10)
\put(0,0){\line(1,0){10}}
\put(0,0){\line(0,1){10}}
\put(10,0){\line(0,1){10}}
\put(0,10){\line(1,0){10}}
\end{picture}}

\newcommand\cellify[1]{\def\thearg{#1}\def\nothing{}%
\ifx\thearg\nothing\vrule width0pt height\cellsize depth0pt%
  \else\hbox to 0pt{\usebox2\hss}\fi%
  \vbox to 10\unitlength{\vss\hbox to 10\unitlength{\hss$#1$\hss}\vss}}

\newlength{\bibitemsep}
\newlength{\bibparskip}
\let\oldthebibliography\thebibliography
\renewcommand\thebibliography[1]{%
  \oldthebibliography{#1}%
  \setlength{\parskip}{\bibitemsep}%
  \setlength{\itemsep}{\bibparskip}%
}

\newcommand\tableau[1]{\vtop{\let\\=\cr
\setlength\baselineskip{-10000pt}
\setlength\lineskiplimit{10000pt}
\setlength\lineskip{0pt}
\halign{&\cellify{##}\cr#1\crcr}}}

\savebox7{
\begin{picture}(10,10)
  \put(5,5){\circle{9.4}}
\end{picture}}

\newcommand{\cirfy}[1]{\def\thearg{#1}\def\nothing{}%
\ifx\thearg\nothing\vrule width0pt height\cellsize depth0pt%
  \else\hbox to 0pt{\usebox7\hss}\fi%
  \vbox to 10\unitlength{\vss\hbox to 10\unitlength{\hss$#1$\hss}\vss}}

\newcommand\cirtab[1]{\vtop{\let\\=\cr
\setlength\baselineskip{-10000pt}
\setlength\lineskiplimit{10000pt}
\setlength\lineskip{0pt}
\halign{&\cirfy{##}\cr#1\crcr}}}

\theoremstyle{plain}

\newtheorem{thm}{Theorem}
\newtheorem*{lem}{Lemma}
\newtheorem*{cor}{Corollary}

\theoremstyle{definition}
\newtheorem*{example}{Example}

\newtheorem*{rem}{Remark}

\theoremstyle{remark}

\advance\textwidth by 1.2in  \advance\oddsidemargin by -.6in \advance\evensidemargin by -.6in

\newcounter{cnt}
 \makeatletter
\def\mydggeometry{\makeatletter\dg@YGRID=1\dg@XGRID=20\unitlength=0.003pt\makeatother}
\makeatother \theoremstyle{remark}

\numberwithin{equation}{section}
\makeatletter
\def\section{\def\@secnumfont{\mdseries}\@startsection{section}{1}%
  \z@{.7\linespacing\@plus\linespacing}{.5\linespacing}%
  {\normalfont\scshape\centering}}
\def\subsection{\def\@secnumfont{\bfseries}\@startsection{subsection}{2}%
  {\parindent}{.5\linespacing\@plus.7\linespacing}{-.5em}%
  {\normalfont\bfseries}}
\def\subsubsection{\def\@subsecnumfont{\bfseries}\@startsection{subsubsection}{3}%
  {\parindent}{.5\linespacing\@plus.7\linespacing}{-.5em}%
  {\normalfont\bfseries}}
  \makeatother

\begin{document}

\title{Two remarks on decomposition numbers of standard modules for quantum affine $\mathfrak{sl}_2$}

\author{Xin Fang}
\address{Lehrstuhl f\"ur Algebra und Darstellungstheorie, RWTH Aachen, Pontdriesch 10-16, 52062 Aachen, Germany}
\email{xinfang.math@gmail.com}
\thanks{X.F. was partially funded by the Deutsche Forschungsgemeinschaft: “Symbolic Tools in Mathematics and their Application” (TRR 195, project-ID 286237555).}

\author{Deniz Kus}
\address{Technical University of Munich, TUM School of Computation, Information and Technology, Department of Mathematics, Boltzmannstr. 3, 85748 Garching bei München, Germany}
\email{deniz.kus@tum.de}
\thanks{D.K. was partially funded by the Deutsche Forschungsgemeinschaft (DFG, German Research Foundation) –  grant 562506224.}

\author{Markus Reineke}
\address{Ruhr-Universität Bochum, Faculty of Mathematics, Universitätsstraße
150, 44780 Bochum, Germany}
\email{markus.reineke@ruhr-uni-bochum.de}
\thanks{}

\maketitle

\begin{abstract}
We use Nakajima's geometric approach to representations of quantum affine algebras and recent results on explicit descriptions of specific canonical basis elements, to derive closed positive formulas for certain decomposition numbers of representations of quantum affine $\mathfrak{sl}_2$. Moreover, we obtain a piecewise-linear closed formula for the $q$-characters of irreducible representations of quantum affine $\mathfrak{sl}_2$.
\end{abstract}

\maketitle

\section{Introduction}

The representation theory of quantum affine algebras has been a rich source of combinatorial and geometric structures. One central problem in this area is the computation of $q$-characters, introduced by Frenkel and Reshetikhin \cite{FR99}, as well as the decomposition numbers for standard modules in terms of simple modules. These numbers encode the multiplicities of simple modules as composition factors and play a crucial role in understanding the Grothendieck ring of the category of finite-dimensional representations - particularly its transition matrices between standard and simple objects.

In this work, we approach this problem via a geometric realization using Nakajima quiver varieties \cite{Nak1,Nak2}. In particular, in \cite{Nak2}, the $t$-analogues of $q$-characters are introduced, and the polynomials - analogs of Kazhdan–Lusztig polynomials for Weyl groups - whose specializations describe the decomposition numbers are defined in terms of stalk cohomology groups of intersection cohomology complexes on quiver varieties. Although these polynomials from \cite{Nak2} can, in principle, be computed by a combinatorial algorithm, this computation is in practice difficult.

The canonical basis is a distinguished basis of the positive part of a quantum group with remarkable properties. The transition matrices between the canonical basis and the PBW basis, corresponding to a reduced expression adapted to a quiver, can again be described in terms of stalk cohomology of intersection cohomology complexes on orbit closures of the corresponding quiver representations. When the quiver is equioriented of type $A_n$, closed formulas for the canonical basis elements corresponding to the so-called sparse representations are given in \cite{FR2}, which are applied to compute the intersection cohomology of the irreducible components of varieties of complexes. Identifying Nakajima quiver varieties of rank one with varieties of complexes and their desingularizations (Theorem \ref{real1}), we derive closed positive formulas for the multiplicities of sparse irreducible representations in standard modules (Theorem \ref{dn1}).

Finally, motivated by the above interplay between rank one quantum affine algebras and higher rank quiver representations, we observe, using \cite{KR}, another such correspondence. Namely, the decomposition of an irreducible representation into prime irreducibles in rank one corresponds exactly to the decomposition of a rigid representation of the equioriented type $A_n$ quiver (where $n$ and the dimension vector of the rigid module is determined by the Drinfeld polynomial) into indecomposables. Using the closed formulas from \cite{KR} naturally leads to explicit formulas for the $q$-characters of all simple modules whose Drinfeld polynomial zeros lie in a single $q$-orbit. It would be interesting to interpret the decomposition of rigid representations for other orientations (piecewise-linear formulas are still available in \cite{KR}) in the language of rank one quantum affine algebras, but this will appear elsewhere.

\textit{Organization of the paper:} In Section~\ref{section2}, we review the necessary background on quantum affine $\mathfrak{sl}_2$, its $q$-characters, and Kirillov–Reshetikhin modules. Section~\ref{section3} recalls the definition of graded Nakajima quiver varieties of rank one, their realization, and their connection to decomposition numbers. In Section~\ref{section4}, we introduce varieties of complexes, and recall the closed formulas for their intersection cohomologies. In Section~\ref{section5}, we relate the varieties of complexes and their desingularizations to Nakajima quiver varieties. Moreover, an explicit formula for decomposition numbers in terms of binomial coefficients is given. Finally, Section~\ref{section6} establishes the link to rigid quiver representations and provides explicit formulas for $q$-characters.

\textit{Acknowledgement: We are grateful to Ryo Fujita for his insightful question, and to David Hernandez for his helpful and stimulating discussions.}

\section{Representations of quantum affine $\mathfrak{sl}_2$}\label{section2}

\subsection{Quantum binomials}
Throughout this paper we denote by $\mathbb{C}$ the field of complex numbers and by $\mathbb{Z}$ (resp. $\mathbb{Z}_{+}$, $\mathbb{N}$) the subset of integers (resp. non-negative, positive integers). We fix $q\in\mathbb{C}^{\times}$ which is not a root of unity and for an indeterminate $v$ we set
$$[n]_v=\frac{v^n-v^{-n}}{v-v^{-1}},\ \ [n]_v!=\prod_{r=1}^n[r]_v,\ \ [0]_v!=1,\ \ \begin{bmatrix} a \\ n \end{bmatrix}_v=\prod_{i=1}^n \frac{v^{a+1-i}-v^{-a-1+i}}{v^i-v^{-i}},\ \  a\in\mathbb{Z},\ n\in\mathbb{N}.$$

Moreover, if $a\in\mathbb{N}$ we set also
$$\binom{a}{n}_t=\Big(v^{n(a-n)}\begin{bmatrix}a\\n\end{bmatrix}_v\Big)\Big|_{v^2=t}\in\mathbb{Z}[t].$$
The usual binomial coefficient (without subscript) is denoted by $\binom{a}{n}$.

\subsection{Quantum affine $\mathfrak{sl}_2$ and $q$-characters}
We recall some known results on finite-dimensional representation of quantum affine algebras associated to $\mathfrak{sl}_2$ and their $q$-characters following \cite{CP91,Le}. The quantum affine algebra $\mathbf{U}_q:=\mathbf{U}_q(\widehat{\mathfrak{sl}}_2)$ in Drinfelds second realization is the associative algebra over $\mathbb{C}$ with generators $x_r^\pm$, $h_m$, $K^{\pm}$, $r,m\in\mathbb{Z}, m\neq 0$ with defining relations
$$KK^-=K^-K=1,\ [h_m,h_k]=[K,h_k]=0,\ \ Kx_k^{\pm}K^-=q^{\pm 2}x_k^{\pm},$$
$$[h_m,x_k^{\pm}]=\pm\frac{1}{m}[2m]_qx_{k+m}^{\pm},$$
$$x_{k+1}^{\pm}x_{\ell}^{\pm}-q^{\pm 2}x_{\ell}^{\pm}x_{k+1}^{\pm}=q^{\pm 2}x_{k}^{\pm}x_{\ell+1}^{\pm}-x_{\ell+1}^{\pm}x_{k}^{\pm},$$
$$[x_k^{+},x_\ell^{-}]=\frac{\phi_{k+\ell}^+-\phi_{k+\ell}^-}{q-q^{-1}},$$
where the $\phi_k^\pm$ are determined by equating coefficients of powers of $u$ in the formula
$$\phi^{\pm}(u)=\sum_{r=0}^{\infty}\phi_{\pm r}^{\pm}u^{\pm r}=K^{\pm}\mathrm{exp}\left(\pm (q-q^{-1})\sum_{k=1}^{\infty} h_{\pm k} u^{\pm k}\right)$$
and $\phi_{\mp r}^{\pm}=0$ for $r\in \mathbb{N}.$ 
The subalgebra generated by $x_k^{\pm}$, $k \in \mathbb{Z}$, is denoted by $\mathbf{U}_q^{\pm}$, and the subalgebra generated by $h_k$ and $K^{\pm}$, $k \in \mathbb{Z} \setminus \{0\}$, is denoted by $\mathbf{U}_q^{0}$. From the defining relations, $\mathbf{U}_q^{0}$ can also be generated by $\phi_k^{\pm}$, $k \in \mathbb{Z}$ and we have a triangular decomposition $\mathbf{U}_q\cong \mathbf{U}_q^{-}\otimes \mathbf{U}_q^{0}\otimes\mathbf{U}_q^{+}$.

\begin{rem}
We omit the central element, since we are interested only in type 1 representations, which by definition means that the center acts trivially and the $K$-eigenvalues belong to $q^{\mathbb{Z}}$.
\end{rem}

Let $\mathcal{P}$ be the free abelian multiplicative group of monomials in the formal variables $\{Y_a: a\in\mathbb{C}^{\times}\}$ and $\mathcal{P}^+\subseteq \mathcal{P}$ the submonoid of dominant monomials, i.e., the monomials with non-negative powers in $Y_a$. We denote the category of finite-dimensional type 1 representations of $\mathbf{U}_q$ by $\mathcal{C}$. The Grothendieck ring of $\mathcal{C}$ is commutative and there is an injective ring homomorphism 
$$\chi_q: K_0(\mathcal{C})\rightarrow \mathcal{Y}:=\mathbb{Z}[Y^{\pm}_a]_{a\in\mathbb{C}^{\times}}$$
known as the $q$-character map. In particular, $\chi_q$ is multiplicative on tensor products. 

To be more precise, the map is given as follows. Decompose a representation $V$ in $\mathcal{C}$ into common generalized eigenspaces
$$V=\bigoplus_{\boldsymbol{\gamma}=(\gamma^{\pm}_r)_{r\in \mathbb{Z}}}V_{\boldsymbol{\gamma}},\ \ V_{\boldsymbol{\gamma}}=\{v\in V: \exists k\in \mathbb{N} \text{ such that }\forall r\in\mathbb{Z},\, (\phi^{\pm}_r-\gamma^{\pm}_r)^kv=0\}.$$
If $V_{\boldsymbol{\gamma}}\neq 0$, it turns out that the generating sequence of the eigenvalues has the following form (see \cite[Section 2.4.]{FR99})
\begin{equation}\label{for1}\sum_{k\geq 0}\gamma^{\pm}_k u^{\pm k}=q^{\mathrm{deg}(Q)-\mathrm{deg}(R)}\frac{Q(q^{-1}u)R(q^{}u)}{Q(q^{}u)R(q^{-1}u)}\end{equation}
as elements of $\mathbb{C}[[u^{\pm}]]$ where $Q(u)=\prod_{a\in \mathbb{C^{\times}}}(1-au)^{k_a}$ and $R(u)=\prod_{a\in \mathbb{C^{\times}}}(1-au)^{r_a}$ are polynomials. Then we have $$\chi_q(V):=\sum_{\boldsymbol{\gamma}}\dim(V_{\boldsymbol{\gamma}})\prod_{a\in\mathbb{C}^{\times}} Y_a^{k_a-r_a}$$
which is a Laurant polynomial in the variables $Y_a^\pm$ for $a\in\mathbb{C}^{\times}$.

The simple objects in $\mathcal{C}$ are, up to isomorphism, parametrized by Drinfeld polynomials in $\mathbb{C}[u]$ with constant term one (see \cite{CP91}) or equivalently by dominant monomials in $\mathcal{Y}$. We denote the simple object associated to $\boldsymbol{\pi}(u)\in \mathbb{C}[u]$ with $\boldsymbol{\pi}(0)=1$ by $V(\boldsymbol{\pi})$. Then $V(\boldsymbol{\pi})$ is generated by an $\ell$-highest weight $v_{\boldsymbol{\pi}}$ satisfying
$$x_r^+v_{\boldsymbol{\pi}}=0,\ \ r\in\mathbb{Z};\ \ \phi^+(u)v_{\boldsymbol{\pi}}=q^{\mathrm{deg}(\boldsymbol{\pi})}\frac{\boldsymbol{\pi}(q^{-1}u)}{\boldsymbol{\pi}(qu)}v_{\boldsymbol{\pi}}.$$

The Kirillov-Reshetikhin representations $W_{n,a}$ for $a\in\mathbb{C}^{\times}$ and $n\geq 1$ are the simple representations with respective Drinfeld polynomials 
\begin{equation}\label{DrKR}
 (1-au)(1-q^2au)\cdots (1-q^{2(n-1)}au).
\end{equation}
Their $q$-characters can be calculated inductively from the $T$-system equations proved in \cite{H06}:
$$\chi_q(W_{n,a})\chi_q(W_{n,aq^2})=\chi_q(W_{n+1,a})\chi_q(W_{n-1,aq^2})+1.$$

\begin{example} 
Given $a\in\mathbb{C}^{\times}$ and $n\geq 1$, the representation $W_{n,a}$ is obtained as follows. Consider the $(n+1)$-dimensional irreducible $\mathbf{U}_q(\mathfrak{sl}_2)$-representation $V_n$ with basis $\{v_0,\dots,v_n\}$ and action
$$Kv_i=q^{n-2i}v_i,\ \ Ev_i=[n-i+1]_qv_{i-1},\ \ Fv_i=[i+1]_qv_{i+1}.$$
There exists a homomorphism of algebras 
$$\mathrm{ev}_{a,n}: \mathbf{U}_q\rightarrow \mathbf{U}_q(\mathfrak{sl}_2)$$
such that 
$$\mathrm{ev}_{a,n}(x_k^+)=(aq^{n-1})^k K^kE,\ \ \mathrm{ev}_{a,n}(x_k^-)=(aq^{n-1})^k FK^k,\ \ k\in\mathbb{Z}.$$
Then $W_{n,a}\cong \mathrm{ev}_{a,n}^*V_n$ is isomorphic to the pull-back of $V_n$ via the above homomorphism.

To see that \eqref{DrKR} is the corresponding Drinfeld polynomial we note that (see \cite[Proposition 4.1] {CP91})
$$x_k^+v_i=a^kq^{k(2(n-i)+1)}[n-i+1]_qv_{i-1},\ \ x_k^-v_i=a^kq^{k(2(n-i)-1)}[i+1]_qv_{i+1}.$$
They imply:
$$\Phi^{\pm}(u)v_i=\frac{q^{n-2i}(1-aq^{-1}u)(1-aq^{2n+1}u)}{(1-aq^{2(n-i)+1}u)(1-aq^{2(n-i)-1}u)}v_i=q^{n-2i}\prod_{r=0}^{n-i-1}\frac{(1-aq^{2r-1}u)}{(1-aq^{2r+1}u)}\prod_{r=0}^{i-1}\frac{(1-aq^{2(n-r)+1}u)}{(1-aq^{2(n-r)-1}u)}v_i.$$
Thus we have written them in the form \eqref{for1} and for $i=0$ we obtain the desired Drinfeld polynomial. 

Moreover, defining $A_a=Y_{aq^{-1}}Y_{aq}$ and the highest weight monomial by $X_{n,a}=Y_aY_{aq^2}\cdots Y_{aq^{2(n-1)}}$, the $q$-character is given by
$$\chi_q(W_{n,a})=\sum_{i=0}^n\prod_{j=1}^{n-i}Y_{aq^{2(j-1)}}\prod_{j=1}^{i}Y^{-1}_{aq^{2(n-i+j)}} = X_{n,a} \sum_{i=0}^n\prod_{j=1}^iA^{-1}_{aq^{2(n-j)+1}}.$$
\end{example}

\begin{rem}
Note that the conventions used here differ from those in \cite{CP91}. 
In particular, the module corresponding to the Drinfeld polynomial~\eqref{DrKR} 
in \cite{CP91} is $\mathrm{ev}_{aq^{-1},n}^*V_n=W_{n, a q^{-1}}$ in the notation adopted above.
\end{rem}

\subsection{$t$-analogue of $q$-characters} 

The $t$-analogue of the $q$-character is a map 
$$\chi_{q,t}: K_0(\mathcal{C})\otimes_{\mathbb{Z}}\mathbb{Z}[t^{\pm}]\rightarrow \hat{\mathcal{Y}}_t:=\mathbb{Z}[t^\pm,V_a,W_a]_{a\in\mathbb{C}^{\times}}$$ defined axiomatically in \cite[Section 3]{Nak2}. Although we will be only interested in the $t=1$ specialization we discuss briefly the notations involved from \cite[Section 2]{Nak2} as they are needed later. 

Following Nakajima, a \emph{monomial} in $\mathbf{m}\in\hat{\mathcal{Y}}_t$ means a monomial in the generators $V_a,W_a$, written as 
\begin{equation}\label{monomial}\mathbf{m}=\prod_{a\in\mathbb{C}^{\times}}V_a^{v_a(m)}W_a^{w_a(m)},
\end{equation}
and similarly for monomials in $\mathcal{Y}_t:=\mathbb{Z}[t^\pm,Y_a^\pm]_{a\in\mathbb{C}^{\times}}$. Given a monomial $\mathbf{m}$ as in \eqref{monomial}, we define 
$$u_a(\mathbf{m})=w_a(\mathbf{m})-v_{q^{-1}a}(\mathbf{m})-v_{qa}(\mathbf{m})$$
and say that $\mathbf{m}$ is $\ell$-\textit{dominant} if $u_a(\mathbf{m})\geq 0$ for all $a\in\mathbb{C}^{\times}$. 
Further, define a $\mathbb{Z}[t^{\pm}]$-linear map $$\hat{\Pi}:\hat{\mathcal{Y}}_t\rightarrow\mathcal{Y}_t,\ \ \hat{\Pi}(\mathbf{m})=t^{-d(\mathbf{m},\mathbf{m})}\prod_aY_a^{u_a(\mathbf{m})},$$ 
where $d(\mathbf{m},\mathbf{m})$ is explicitly defined in \cite[Equation (2.1)]{Nak2}, but we will never need the definition. Then, on $K_0(\mathcal{C})$ we have an equality $\chi_q=\hat{\Pi}\circ\chi_{q,t}{\big|}_{t=1}$. 

We have a correspondence between Drinfeld polynomials and monomials in $\hat{\mathcal{Y}}_t$ in the variables $W_a$ given by
$$\boldsymbol{\pi}(u)=(1-a_1u)\cdots (1-a_ru)\mapsto e^{\boldsymbol{\pi}}:=W_{a_1}\cdots W_{a_r}$$
Similarly, we have an identification between quotients of Drinfeld polynomials $R(u)\in\mathbb{C}(u)$ and monomials in $\mathcal{Y}$: the monomial corresponding to $R(u)$ is denoted by $\tilde{e}^R\in \mathcal{Y}$.
Note that $\tilde{e}^R\in\mathcal{P}^+$ if and only if $R(u)$ is a polynomial.

Finally, we define a partial order among the monomials in $\hat{\mathcal{Y}}_t$ and $\mathcal{Y}$. Given two monomials $\mathbf{m},\tilde{\mathbf{m}}\in \hat{\mathcal{Y}}_t$ (resp. $\mathbf{m},\tilde{\mathbf{m}}\in \mathcal{Y}$) we say that $\mathbf{m}\leq \tilde{\mathbf{m}}$ if $\mathbf{m}/\tilde{\mathbf{m}}$ is a monomial in the variables $V_a$  (resp.  in the variables $A^{-1}_a$). Note that $\hat{\Pi}(V_a)=A^{-1}_a$, so the partial order is preserved by $\hat{\Pi}$.

\subsection{$q$-strings and standard modules} 

We quickly recall the realization of simple objects and the algebraic definition of standard modules. 

A finite-set of non-zero complex numbers of the form 
$$S_{n,a}=\{a,q^2a,\ldots,q^{2(n-1)}a\},\  a\in\mathbb{C}^{\times},\  n\in\mathbb{N}$$
is called a $q$-\textit{string}. Two $q$-strings $S_{n,a}$, $S_{m,b}$ are called in special position if they are not contained in each other and their union is again a $q$-string: that is,
$$\frac{b}{a}\in\left\{q^{2(n-p+1)}, q^{-2(m-p+1)}: 1\leq p\leq \min\{m,n\}\right\}.$$
Otherwise they are called in general position. 

We have a bijection between Drinfeld polynomials and finite multisets in $\mathbb{C}^{\times}$ (that is, functions $S:\mathbb{C}^{\times}\rightarrow\mathbb{Z}_+$ with finite support) by associating to $S$ the polynomial $\prod_{a\in \mathbb{C}^{\times}}(1-au)^{S(a)}.$ Every finite multiset $S$ admits a unique decomposition into a union of $q$-strings $S_{n_1,a_1},\ldots,S_{n_s,a_s}$ which are pairwise in general position. We will comment more on this decomposition and its connection to rigid representations of quivers in Section~\ref{section6}. If $\boldsymbol{\pi}$ is the Drinfeld polynomial associated to $S$, then
$$V(\boldsymbol{\pi})=W_{n_1,a_1}\otimes W_{n_2,a_s}\otimes\cdots \otimes W_{n_s,a_s}.$$
Hence, the finite-dimensional irreducible representations are tensor products of Kirillov--Reshetikhin representations.

Another important class of representations, introduced by Nakajima geometrically and interpreted algebraically in \cite{VV02}, is given by the so-called \textit{standard modules} $M(\boldsymbol{\pi})$, which we now recall. We write a Drinfeld polynomial $\boldsymbol{\pi}$ in the form    
$$\boldsymbol{\pi}(u)=(1-a_1u)\cdots (1-a_ru)$$
such that $k<\ell \implies a_{\ell}/a_k\notin q^{\mathbb{N}}$. This is possible since the directed graph with vertices $\{1,\dots,r\}$ and arrows $i\rightarrow j$ if $a_i\in q^{\mathbb{N}}a_j$ is acyclic: such a graph admits hence a topological sorting. The corresponding tensor product 
$$M(\boldsymbol{\pi}):=W_{1,a_1}\otimes  \cdots \otimes W_{1,a_r}$$
is the standard module associated to $\boldsymbol{\pi}$. It does not depend on the topological sorting since (see \cite[Theorem 5.1]{Chabraid})
$$a/b\notin q^{\mathbb{Z}}\text{ implies } W_{1,a}\otimes W_{1,b}\cong W_{1,b}\otimes W_{1,a}.$$
The geometric definition of these modules in terms of quiver varieties will be given in the next section.
\begin{rem}
The ordering in the definiton of $M(\boldsymbol{\pi})$ is important. For example, $W_{1,aq^2}\otimes W_{1,a}$ is a standard module whereas $W_{1,a}\otimes W_{1,q^2a}$ is not even a cyclic module.
\end{rem}

\section{Quiver varieties and varieties of complexes}\label{section3}

In this section, we introduce notations on graded Nakajima quiver varieties, later in Section \ref{section5}, those of type $A_1$ will be identified with varieties of complexes and their desingularizations. The main goal is to recall formulae on $t$-analogue of the $q$-character of standard modules, expressed in terms of intersection cohomologies of quiver varieties. 

\subsection{Graded Nakajima quiver varieties} 
 
We follow closely the notation of \cite[Section 4]{Nak2} with $\epsilon$ replaced by $q^{-1}$ throughout.  To specify the quiver of type $A_1$ we define $I=\{i\}$ and $E=\emptyset$. For every $a\in\mathbb{C}^{\times}$, we choose  complex vector spaces $V(a)=V_i(a)$ and $W(a)=W_i(a)$, so that the direct sum of all such spaces is finite-dimensional. We then consider the finite-dimensional vector space 
 $$\mathbf{M}^{\bullet}=\bigoplus_{a \in\mathbb{C}^{\times}}{\mathrm{Hom}}(W(a),V(qa))\oplus\bigoplus_{a\in\mathbb{C}^{\times}}{\mathrm{Hom}}(V(a),W(qa))$$
 and denote an element of this space as $(\alpha,\beta)=((\alpha_a)_{a\in\mathbb{C}^{\times}},(\beta_a)_{a\in\mathbb{C}^{\times}})$ for $\alpha_a\in{\mathrm{Hom}}(W(a),V(qa))$ and $\beta_a\in{\mathrm{Hom}}(V(a),W(qa))$. The reductive group
 $$G_V=\prod_{a\in \mathbb{C}^{\times}}{\mathrm{GL}}(V(a))$$
 acts on $\mathbf{M}^{\bullet}$ by
 $$(g_a)_a\cdot((\alpha_a)_a,(\beta_a)_a)=((g_{qa}\alpha_a)_a,(\beta_a g_a^{-1})_a).$$
We have a (momentum-type) map
$$\mu:\mathbf{M}^{\bullet}\rightarrow\bigoplus_{a\in \mathbb{C}^{\times}}{\mathrm{Hom}}(V(a),V(q^{2}a)), \ \ \ \mu((\alpha_a)_a,(\beta_a)_a)=(\alpha_{qa}\beta_a)_a,$$
 which is $G_V$-equivariant in a natural way and preserves the subvariety $\mu^{-1}(0)$. 
 
A point $(\alpha,\beta)\in\mu^{-1}(0)$ is said to be \emph{stable} if $\beta_a$ is injective for all $a\in \mathbb{C}^{\times}$ and denote by $\mu^{-1}(0)^{s}$ the open subset of stable points in $\mu^{-1}(0)$. The \textit{graded Nakajima quiver varietes} are defined as GIT quotients:
$$\mathfrak{M}_0^{\bullet}(V,W):=\mu^{-1}(0)//G_V,\ \ \ \mathfrak{M}^{\bullet}(V,W):=\mu^{-1}(0)^{s}/G_V.$$
The first variety parametrizes the closed $G_V$-orbits on $\mu^{-1}(0)$ and is realized as $\mathrm{Spec}$ of the ring of $G_V$-invariant functions on $\mu^{-1}(0)$. The second variety, being the $\mathrm{Proj}$ of the ring of semi-invariants, coincides in this setting with the set-theoretical quotient. Note also that the first variety is non-empty whereas the second one may be empty. 
 
For a stable point $(\alpha,\beta)\in\mu^{-1}(0)$, we denote its $G_V$-orbit in $\mathfrak{M}^{\bullet}(V,W)$ by $[\alpha,\beta]$. Moreover, if the $G_V$-orbit through $(\alpha,\beta)\in\mu^{-1}(0)$ is closed, then the corresponding point in $\mathfrak{M}_0^{\bullet}(V,W)$ will be also denoted by $[\alpha,\beta]$. We define the (possibly empty) open subset
 $$\mathfrak{M}_0^{\bullet,{\rm reg}}(V,W)=\{[\alpha,\beta]\in\mathfrak{M}_0^{\bullet}(V,W): (\alpha,\beta) \ \text{ has trivial stabilizer in $G_V$}\}.$$
There is a natural projective morphism $$\pi:\mathfrak{M}^{\bullet}(V,W)\rightarrow\mathfrak{M}_0^{\bullet}(V,W)$$ 
which maps the orbit $[\alpha,\beta]$ to the unique closed $G_V$-orbit in its closure in $\mu^{-1}(0)$.
We finally define the projective variety $\mathfrak{L}^{\bullet}(V,W):=\pi^{-1}(0).$
Hence, points in $\mathfrak{L}^{\bullet}(V,W)$ correspond to $G_V$-orbits in $\mu^{-1}(0)^{s}$ whose orbit closure in $\mathbf{M}^{\bullet}$ contains $0$.

Since $q$ is not a root of unity, the variety $\mathbf{M}^{\bullet}$, defined as a direct sum over $a\in\mathbb{C}^{\times}$, naturally decomposes along the orbits of the action of $q^2$ on $\mathbb{C}^{\times}$ given by multiplication. Therefore, we will fix $a\in\mathbb{C}^{\times}$ in the following and assume without further comment that $V(qb)=0=W(b)$ if $b\neq q^{2k}a$ for $k\in\mathbb{Z}$. An element in this $q^2$-orbit in $\mathbf{M}^{\bullet}$ can be depicted as follows:
$$\begin{array}{ccccccccccc}
&W(q^{2}a)&&&&W(a)&&&&W(q^{-2}a)&\\
\swarrow&&\nwarrow&&\swarrow&&\nwarrow&&\swarrow&&\nwarrow\\
&&&V(qa)&&&&V(q^{-1}a)&&&\end{array}$$
The coordinate ring of $\mathfrak{M}^{\bullet}_0(V,W)$ is generated by the $G_V$-invariant functions on $\mu^{-1}(0)$ given by
$$(\alpha,\beta)\mapsto \varphi(\beta_{q^{2k+1}a}\alpha_{q^{2k}a}),\ \ \ k\in\mathbb{Z},\ \ \varphi\in \mathrm{Hom}(W(q^{2k}a),W(q^{2(k+1)}a))^*.$$

We give a description of the quiver varieties:

\begin{thm}\label{real1} 
The variety $\mathfrak{M}^{\bullet}(V,W)$ is isomorphic to the subvariety of
$$\bigoplus_{k\in\mathbb{Z}}{\mathrm{Hom}}(W(q^{2k}a),W(q^{2(k+1)}a))\times\prod_{k\in\mathbb{Z}}{\rm Gr}_{\dim V(q^{2k-1}a)}(W(q^{2k}a))$$
of tuples $((f_k)_k,(U(q^{2k}a)\subseteq W(q^{2k}a))_k)$ such that $f_k(W(q^{2k}a))\subseteq U(q^{2(k+1)}a)$ and $f_k(U(q^{2k}a))=0$ for all $k\in\mathbb{Z}$. 

The affine variety $\mathfrak{M}^{\bullet}_0(V,W)$ is isomorphic to the subvariety 
$$\left\{(f_k)_{k\in\mathbb{Z}}\in \bigoplus_{k\in\mathbb{Z}}{\mathrm{Hom}}(W(q^{2k}a),W(q^{2(k+1)}a)): f_{k+1}f_k=0,\ \ {\rm rk}(f_k)\leq \dim V(q^{2k+1}a),\ \ \forall k\in\mathbb{Z}\right\}.$$ 

Under these isomorphisms, the map $\pi$ is given by forgetting the subspaces. 
\end{thm}

\begin{proof}
Note that $a\in\mathbb{C}^\times$ has been fixed. The first isomorphism is given by 
$$(\alpha,\beta)\mapsto ((f_k)_{k\in\mathbb{Z}},(U(q^{2k}a))_{k\in\mathbb{Z}})$$
where $f_k:=\beta_{q^{2k+1}a}\alpha_{q^{2k}a}$ and $U(q^{2k}a):=\mathrm{im}\beta_{q^{2k-1}a}$. The conditions on $f_k$ and $U(q^{2k})$ follow from $(\alpha,\beta)\in\mu^{-1}(0)$ is stable. Now it suffices to pass to the $G_V$-orbits.

The description of $\mathfrak{M}^{\bullet}_0(V,W)$ follows similarly by noticing the property that $f_k$ admits a canonical factorization through  $V(q^{2k+1}a)$ and is contained in the orbit closure of any other factorization.

\end{proof}

We will mostly work with the above realizations without further comment. If the variety $\mathfrak{M}^{\bullet}(V,W)$ is non-empty, the projection of $\mathfrak{M}^{\bullet}(V,W)$ to the product of Grassmannians is a vector bundle and hence the variety is irreducible and smooth of dimension:
$$\mathrm{dim}\ \mathfrak{M}^{\bullet}(V,W)=\sum_{k\in\mathbb{Z}}\big(\mathrm{dim}W(q^{2k}a)-\mathrm{dim}V(q^{2k-1}a)\big)\big(\mathrm{dim}V(q^{2k+1}a)+\mathrm{dim}V(q^{2k-1}a)\big).$$

As in \cite[Section 4]{Nak2}, to graded vector spaces $V$, $W$ as above, we associate monomials in $\hat{\mathcal{Y}}_t$ as follows:
$$e^W=\prod_aW_a^{\dim W(a)},\; e^V=\prod_aV_a^{\dim V(a)}.$$
Note that $e^W=e^{\boldsymbol{\pi}_W}$ for the polynomial $\boldsymbol{\pi}_W(u)=\prod_{a\in\mathbb{C}^{\times}}(1-au)^{\text{dim}W(a)}$. We have 
$$\mathfrak{M}^{\bullet,{\rm reg}}_0(V,W)\neq \emptyset \iff e^V\cdot e^W \text{ is $\ell$-dominant and there exists a stable point in $\mu^{-1}(0)$}.$$
Also, in this case, $\mathfrak{M}^{\bullet,{\rm reg}}_0(V,W)$ is smooth of dimension equal to $\mathrm{dim }\ \mathfrak{M}^{\bullet}(V,W)$. In fact, $$\mathfrak{M}^{\bullet,{\rm reg}}_0(V,W)\subseteq \mathfrak{M}^{\bullet}_0(V,W)$$ is given by tuples $(f_k)_k$ such that ${\rm rk}(f_k)=\dim V(q^{2k+1}a)$ for all $k\in\mathbb{Z}$.

We define
$$\mathfrak{M}_0^{\bullet}(W):=\bigsqcup_{[V]}\mathfrak{M}_0^{\bullet,{\rm reg}}(V,W),\ \ \mathfrak{M}^{\bullet}(W):=\bigsqcup_{[V]}\mathfrak{M}^{\bullet}(V,W),\ \ \mathfrak{L}^{\bullet}(W):=\bigsqcup_{[V]}\mathfrak{L}(V,W)$$ 
where $[V]$ in the disjoint union denotes the isomorphism class as a graded vector space. In fact, if for any $a\in\mathbb{C}^\times$, $V(a)\subseteq V'(a)$, there exists a closed embedding $\mathfrak{M}_0^{\bullet}(V,W)\subseteq \mathfrak{M}_0^{\bullet}(V',W)$. The $\mathfrak{M}_0^{\bullet}(W)$ defined above coincides with the limit, and the decomposition above gives a stratification. Each stratum is precisely the orbit for the natural action of $G_W=\prod_{a\in\mathbb{C}^\times}{\mathrm{GL}}(W(a))$ on $\mathfrak{M}_0^{\bullet}(W)$ by change of bases and hence connected.

\subsection{Standard module and decomposition number} 

Using these geometries, we can give a geometric interpretation of decomposition numbers following \cite[Theorem 8.6]{Nak2}. 

First we recall the geometric definition of standard modules. Let $\boldsymbol{\pi}=\prod_{a\in\mathbb{C}^{\times}}(1-au)^{d_a}$ be a Drinfeld polynomial and choose a $\mathbb{C}^\times$-graded vector space $W$ so that $\text{dim }W(a)=d_{a}$. In particular, $e^{\boldsymbol{\pi}}=e^W$ and we have a bijection between isomorphism classes of $\mathbb{C}^{\times}$-graded vector spaces and monomials $\mathbf{m}\in \hat{\mathcal{Y}}_t$ satisfying $\mathbf{m}\leq e^W$. The map is given by $[V]\mapsto e^V\cdot e^W$ and denote the space corresponding to $\mathbf{m}$ by $V_{\mathbf{m}}$.

The standard module $M(\boldsymbol{\pi})$ is given as the Borel-Moore homology $\mathrm{H}_{*}(\mathfrak{L}^{\bullet}(W),\mathbb{C})$ with complex coefficients, equipped with an action of the quantum affine algebra by the convolution product \cite{Nak1}.
Moreover, the $t$-analogue of the $q$-character is given by $$\chi_{q,t}(M(\boldsymbol{\pi}))=\sum_{[V],k}t^k\ \mathrm{dim}\, \mathrm{H}_k(\mathfrak{L}^{\bullet}(V,W))\ e^V\cdot e^W$$
and thus applying $\hat{\Pi}$, 
$$\chi_{q}(M(\boldsymbol{\pi}))=\sum_{[V]}\chi(\mathfrak{L}^{\bullet}(V,W))\ \prod_{a\in\mathbb{C}^{\times}}Y_a^{d_a}A_a^{-\text{dim}V(a)}$$ 
where $\chi(\cdot)$ denotes the Euler characteristic. In particular, by Theorem~\ref{real1} we have that $\mathfrak{L}^{\bullet}(V,W)$ is just a product of Grassmannians and thus the Euler characteristic of $\mathfrak{L}^{\bullet}(V,W)$ is just given by products of binomial coefficients giving a closed formula.

To state the multiplicity formula, recall that $\mathfrak{M}^{\bullet,{\mathrm{reg}}}_0(V,W)$ is connected and let $\mathbf{IC}(\overline{\mathfrak{M}^{\bullet,\text{reg}}_0(V,W)})$ be the intersection cohomology complex associated with the constant local system $\mathbb{C}_{\mathfrak{M}^{\bullet,\text{reg}}_0(V,W)}$. Given two monomials $\mathbf{m},\tilde{\mathbf{m}}\leq e^{\boldsymbol{\pi}}$ in $\hat{\mathcal{Y}}_t$ we fix a point $x_{\mathbf{m}}\in \mathfrak{M}^{\bullet,\text{reg}}_0(V_{\mathbf{m}},W)$ and define 
$$Z_{\mathbf{m},\tilde{\mathbf{m}}}(t)=\sum_k\dim \mathrm{H}^k(i_{x_{\mathbf{m}}}^!\mathbf{IC}(\overline{\mathfrak{M}^{\bullet,\text{reg}}_0(V_{\tilde{\mathbf{m}}},W)}))\ t^{D-k}.$$ where again $D$ is a constant depending on $\mathbf{m}$ which we will not need. Note also that $Z_{\mathbf{m},\tilde{\mathbf{m}}}(t)\neq 0$ implies $\tilde{\mathbf{m}}\leq \mathbf{m}$.
The following is proved in \cite[Theorem 8.6]{Nak2}. We keep the same notation as above.
\begin{thm}\label{tnak} The multiplicity of a simple object in the standard module is given by
$$[M(\boldsymbol{\pi}):V(\widetilde{\boldsymbol{\pi}})]=\sum_{\tilde{\mathbf{m}}}Z_{e^{\boldsymbol{\pi}},\tilde{\mathbf{m}}}(1),$$ where the sum ranges over all monomials $\tilde{\mathbf{m}}\in \hat{\mathcal{Y}}_t$ such that $\tilde{\mathbf{m}}\leq e^{\boldsymbol{\pi}}$ and $\left.\hat{\Pi}(\tilde{\mathbf{m}})\right|_{t=1}=\tilde{e}^{\tilde{\boldsymbol{\pi}}}$.
\qed
\end{thm}

\subsection{Quantum affine $\mathfrak{sl}_2$}\label{subsec33} 

To make this concrete in our setting, assume that Drinfeld polynomials \begin{equation}\label{drdo}\boldsymbol{\pi}(u)=\prod_{i=0}^{n-1}(1-q^{2i}au)^{w_i},\ \ \widetilde{\boldsymbol{\pi}}(u)=\prod_{i=0}^{n-1}(1-q^{2i}au)^{h_i}\end{equation} are given. We choose vector spaces 
\begin{equation}\label{spch}
W(a), W(q^2a),\dots, W(q^{2(n-1)}a), V(q^{-1}a),V(qa),\dots, V(q^{2n-1}a)\end{equation}
satisfying $V(q^{-1}a)=V(q^{2n-1}a)=\{0\}$, $\mathrm{dim} W(q^{2i}a)=w_i$ for $i=0,\dots,n-1$ and 
$$w_i-\dim V(q^{2i-1}a)-\dim V(q^{2i+1}a)=h_i,\ \ i=0,\dots,n-1.$$

This corresponds in Theorem~\ref{tnak} exactly for choosing monomials $\tilde{\mathbf{m}}$ satisfying $\left.\hat{\Pi}(\tilde{\mathbf{m}})\right|_{t=1}=\tilde{e}^{\tilde{\boldsymbol{\pi}}}$; recall that $Z_{e^{\boldsymbol{\pi}},\tilde{\mathbf{m}}}(1)\neq 0$ implies $\tilde{\mathbf{m}}\leq e^{\boldsymbol{\pi}}$. If the above system has a solution, it has to be unique. Note that the monomial $e^{\boldsymbol{\pi}}$ does not involve variables $V_a$, thus for the element $x_{e^{\boldsymbol{\pi}}}$ in the definition of $Z_{e^{\boldsymbol{\pi}},\tilde{\mathbf{m}}}$ in Theorem~\ref{tnak}, we take the point $0\in\mathfrak{M}^{\bullet,\mathrm{reg}}_0(0,W)$. Therefore, for the choices made above
\begin{equation}\label{Eq:Multiplicity}
[M(\boldsymbol{\pi}):V(\widetilde{\boldsymbol{\pi}})]=\sum_k\dim \mathrm{H}^k(i_{0}^!\mathbf{IC}(\overline{\mathfrak{M}^{\bullet,\mathrm{reg}}_0(V,W)}))=\sum_k\dim \mathcal{H}_0^k(\mathbf{IC}(\overline{\mathfrak{M}^{\bullet,\mathrm{reg}}_0(V,W)})).
\end{equation}
Our aim will be to give a closed formula for the above expression.

\section{Perverse sheaves on varieties of complexes}\label{section4}

In this section, we will recall some results on varieties of complexes and canonical bases of quantum groups following \cite{FR1,FR2}.

We fix complex vector spaces $W^0,\ldots,W^{n-1}$ of dimensions $w_0,\ldots,w_{n-1}$ respectively, and denote by ${\rm Com}(W_*)$ the affine subvariety of
$${\mathrm{Hom}}(W^0,W^1)\times\cdots\times {\mathrm{Hom}}(W^{n-2},W^{n-1})$$
consisting of tuples of maps $(f_0,\dots,f_{n-2})$ such that $f_{i+1}\circ f_i=0$ for all $i=1,\dots,n-2$ where we set $f_{n-1}=0$. 

The group $G:=\mathrm{GL}(W^0)\times\ldots\times \mathrm{GL}(W^{n-1})$ acts on ${\rm Com}(W_*)$ via change of basis, and the orbits $\mathcal{O}({\mathbf{r}})$ are given by complexes of maps of fixed ranks $\mathbf{r}=(r_0,r_1,\ldots,r_{n-2})$,
$$\mathcal{O}({\mathbf{r}}):=\{(f_0,\ldots,f_{n-2})\in{\rm Com}(W_*)\, :\, {\rm rank}(f_i)=r_i,\; i=0,\ldots,n-2\}.$$
Equivalently, with the convention $r_{-1}=r_{n-1}=0$, we can describe these orbits as complexes with fixed Betti numbers, since
$$h_i=h_i({\mathbf{r}})=\dim \mathrm{H}^i((W_*,f_*))=\dim{\mathrm{Ker}}(f_i)-\dim {\mathrm{Im}}(f_{i-1})=w_i-r_{i-1}-r_i,\ \ 0\leq i\leq n-1.$$

We have $\mathcal{O}({\bf r}')\subseteq\overline{\mathcal{O}({\bf r})}$ if and only if ${\mathbf{r}}'\leq{\mathbf{r}}$ componentwise. Consequently, the irreducible components of ${\rm Com}(W_*)$ are the closures of the maximal orbits, which correspond to rank tuples ${\bf r}$ for which the set $\Omega({\mathbf{r}}):=\{i=0,\ldots,n-1\, :\, h_i({\mathbf{r}})\not=0\}$ is sparse, i.e. it contains no two consecutive indices.

\begin{rem}
The variety of complexes ${\rm Com}(W_*)$ can be naturally identified with a closed subvariety of the representation space $R_{\mathbf{d}}(Q)$ of the equioriented quiver $Q$ of type $A_n$ where $\mathbf{d}=(w_0,\dots,w_{n-1})$ and $w_i=\text{dim } W^i$. More precisely, ${\mathrm{Com}}(W_*)$ consists of those representations of $Q$ that decompose as direct sums of indecomposable representations supported at one or two vertices.
\end{rem}

Let ${\mathbf{IC}}(\overline{\mathcal{O}({\bf r})})$ be the $\ell$-adic intersection cohomology complex on the orbit closure $\overline{\mathcal{O}({\mathbf{r}})}$ and denote the stalks of the cohomology sheaves of ${\mathbf{IC}}(\overline{\mathcal{O}({\mathbf{r}})})$ over a point $f_*\in \mathcal{O}({\mathbf{r}}-{\mathbf{k}})\subseteq \overline{\mathcal{O}({\bf r})}$ by $\mathcal{H}^i_{f_*}({\mathbf{IC}}(\overline{\mathcal{O}({\mathbf{r}})}))$, where the tuple $\mathbf{k}:=(k_0,k_1,\ldots,k_{n-2})$ satisfies $\mathbf{k}\leq\mathbf{r}$ componentweise. The Poincaré polynomial of $\mathcal{H}^{*}_{f_*}({\mathbf{IC}}(\overline{\mathcal{O}({\mathbf{r}})}))$ can be computed from a particular coefficient in the expansion of a suitable canonical basis element in the PBW basis (with respect to a reduced word adapted to the equioriented type \(A_n\) quiver).

Using this approach, the main result of \cite{FR2}, namely Theorem 5.1, can be stated as follows:

\begin{thm}\label{t1} If $\overline{\mathcal{O}({\bf r})}$ is an irreducible component of ${\rm Com}(W_*)$ (i.e. $\Omega({\mathbf{r}})$ is sparse) and $f_*\in\mathcal{O}({\bf r}-{\bf k})\subseteq \overline{\mathcal{O}({\bf r})}$, then

\begin{align*}\sum_i&\dim\mathcal{H}^{i}_{f_*}({\bf IC}(\overline{\mathcal{O}({\bf r})})){\text{\Large t}}^{\frac{i}{2}}&\\&
=\sum_{(0\leq a_i\leq \mathrm{min}\{k_{i-1},k_i\})_{i\in \Omega({\bf r})}}{\text{\Large t}}^{\sum_{i\in\Omega(\mathbf{r})}(h_i(\mathbf{r})+a_i)a_i}\cdot\prod_{i\in\Omega({\bf r})}{k_i\choose k_i-a_i}_t{k_{i-1}\choose a_i}_t\cdot\prod_{i\not\in\Omega({\bf r})}{k_{i-1}+k_i\choose k_i}_t\end{align*}
with the usual convention $k_{-1}=k_{n-1}=0.$
\qed
\end{thm}

\section{Decomposition Numbers of standard modules}\label{section5}

In this section we apply the constructions and results above to determine certain multiplicities of irreducible representations in standard modules.

We follow the notation of Section~\ref{subsec33}: Fix Drinfeld polynomials 
$\boldsymbol{\pi}$ and $\widetilde{\boldsymbol{\pi}}$ as in \eqref{drdo} and vector spaces as in \eqref{spch}. We set $\mathbf{r}=(r_0,r_1,\dots,r_{n-2})$ where $r_i=\mathrm{dim} V(q^{2i+1}a)$ with the usual convention $r_{-1}=r_{n-1}=0$ and consider ${\rm Com}(W_*)$ for the family $W^i=W(q^{2i}a)$, $i=0,\dots,n-1$. Using Theorem~\ref{real1} and the discussions in Section~\ref{section4}  we have
$$\mathfrak{M}^{\bullet}_0(V,W)\cong \overline{\mathcal{O}(\mathbf{r})}$$
and
$$\mathfrak{M}^{\bullet}(V,W)\cong\widetilde{\text{Com}}(W_*,{\mathbf{r}}):=\big\{(f_*,(U_i\subseteq W^{i+1}))\in{\rm Com}(W_*)\times\prod_{i=0}^{n-2}{\rm Gr}_{r_i}(W^{i+1})\, :$$
$${\rm Im}(f_{i})\subseteq U_i\subseteq{\rm Ker}(f_{i+1}),\, i=0,\dots,n-2\big\}.$$
Moreover, asking for trivial stabilizer in $G_V$ is equivalent to requiring $f_0,\ldots,f_{n-2}$ having full rank; we have hence
\begin{equation}\label{identif}
\mathfrak{M}^{\bullet,{\rm reg}}_0(V,W)\cong \mathcal{O}(\mathbf{r}).
\end{equation}

The obvious forgetful map $\pi$ to ${\rm Com}(W_*)$ is a desingularization of the closure of $\mathcal{O}({\bf r})$ (if the latter is non-empty). Moreover, note that the exponents of $\widetilde{\boldsymbol{\pi}}$ are by construction given by $h_i=h_i(\mathbf{r})$.

\begin{thm}\label{dn1} 
Let $\boldsymbol{\pi}$ and $\widetilde{\boldsymbol{\pi}}$ as in \eqref{drdo} and set
$$r_i:=\sum_{j=0}^i(-1)^{i+j}(w_{j}-h_{j}),\ \ i=0,\dots,n-2.$$ 
\begin{enumerate}
    \item If there exists an element $i\in\{0,\dots,n-2\}$ such that $r_i < 0$ or $r_{n-2}\neq w_{n-1}-h_{n-1}$, then we have $[M(\boldsymbol{\pi}):V(\widetilde{\boldsymbol{\pi}})]=0$.
    \item Otherwise, the following formula holds if $\Omega(\mathbf{r})$ is sparse:
    $$[M(\boldsymbol{\pi}):V(\widetilde{\boldsymbol{\pi}})]=\sum_{(0\leq a_i\leq \mathrm{min}\{r_{i-1},r_i\})_{i\in \Omega}}\prod_{i\in\Omega}{r_i\choose r_i-a_i}{r_{i-1}\choose a_i}\cdot\prod_{i\not\in\Omega}{r_{i-1}+r_i\choose r_i}.$$
\end{enumerate}
\end{thm}

\begin{proof}
Following the lines in Section~\ref{subsec33}, the condition $r_i<0$ means that there exists no $\widetilde{\mathbf{m}}$ with $\hat{\Pi}(\tilde{\mathbf{m}})=\tilde{e}^{\tilde{\boldsymbol{\pi}}}$ and the claim follows with Theorem~\ref{tnak}. Otherwise, we also have $r_{i-1}+r_i\leq w_i$ and $\mathbf{r}$ is in fact a rank tuple of some complex and sparsity guarantees that the orbit closure is an irreducible component. The theorem follows now by combining Theorem~\ref{tnak} (more precisely, the concretization in \eqref{Eq:Multiplicity} in Section~\ref{subsec33}) and Theorem~\ref{t1} for the choice $f_{*}=0$.
\end{proof}
In fact, via the above identifications, the coefficients appearing in Theorem~\ref{t1} can be interpreted as decomposition numbers associated with a Jantzen filtration, see \cite{G96,Nak1,Nak2,HF26} for further details.
\begin{example}
We consider $\widetilde{\boldsymbol{\pi}}(u)=(1-aq^2u)^2$ and $\boldsymbol{\pi}(u)=(1-au)(1-aq^2u)^3$ and obtain $r_0=1, r_1=0,$ and according to Theorem~\ref{dn1} we have $[M(\boldsymbol{\pi}):V(\widetilde{\boldsymbol{\pi}})]=1$.
\end{example}
As a corollary, keeping the same notation as above, we derive the following criterion:
\begin{cor}
The irreducible module $V(\widetilde{\boldsymbol{\pi}})$ appears as subquotient of $M(\boldsymbol{\pi})$ if and only if $\mathbf{IC}(\overline{\mathcal{O}({\bf r})})$ (up to some shift) appears as direct summand of the complex of constructible sheaves $R\pi_*\mathbb{C}_{\mathfrak{M}^{\bullet}(W)}$.
\end{cor}

\begin{proof} This follows along the lines of the proof of \cite[Theorem 8.6]{Nak2}. From the decomposition theorem we have a linear isomorphism (ignoring gradings)
  $$R\pi_*\mathbb{C}_{\mathfrak{M}^{\bullet}(W)}\cong \bigoplus_{V} L_V\otimes \mathbf{IC}(\overline{\mathfrak{M}_0^{\bullet,\mathrm{reg}}(V,W)})$$
and each vector spaces $L_V$ is an irreducible $\mathbf{U}_q$-module whose Drinfeld polynomial $\boldsymbol{\pi}_V$ is determined by $\left.\hat{\Pi}(e^V\cdot e^W)\right|_{t=1}=\tilde{e}^{\boldsymbol{\pi}_V}$. Since $M(\boldsymbol{\pi})$ can be identified with $\mathcal{H}_0^*(R\pi_*\mathbb{C}_{\mathfrak{M}^{\bullet}(W)})$ we get the desired claim with \eqref{identif}.
\end{proof}
\section{Tensor products and rigid quiver representations}\label{section6}

Let $Q$ be a Dynkin quiver with set of vertices $Q_0$, and let ${\bf d}\in\mathbb{Z}_+Q_0$ be a dimension vector for $Q$. Then there exists a unique representation (up to isomorphism) $V$ of dimension vector ${\bf d}$ such that ${\rm Ext}^1_Q(V,V)=0$ (called a rigid representation). Equivalently, if $V=U_1\oplus\dots\oplus U_s$ is the direct sum decomposition of $V$ into indecomposable representations $U_i$, then ${\rm Ext}_Q^1(U_i,U_j)=0$ for all $i,j=1,\dots,s$. 

In particular, let $Q$ be the linearly oriented quiver of type $A_n$ given by
$$1\longrightarrow 2\longrightarrow\dots\longrightarrow n.$$
The indecomposables $U_{i,j}$ of dimension vector $e_i+e_{i+1}+\cdots+e_j$ are parametrized by pairs $1\leq i\leq j\leq n$. From the known representation theory of $Q$ (in particular, from the Auslander-Reiten quiver) we see easily that ${\mathrm{Ext}}_Q^1(U_{i,j},U_{r,s})\neq 0$ if and only if $i+1\leq r\leq j+1\leq s\leq n$. Hence 
$${\mathrm{Ext}}_Q^1(U_{i,j},U_{r,s})={\mathrm{Ext}}_Q^1(U_{r,s},U_{i,j})=0\iff S_{j-i+1,q^{2i}a},\ S_{s-r+1,q^{2r}a} \text{ are in general position}$$
where we recall that $S_{j-i+1,q^{2i}a}=\{q^{2i}a,q^{2i+2}a,\ldots,q^{2j}a\}.$
From this, we conclude:

\begin{lem}\label{gp} For ${\bf d}\in\mathbb{Z}_+^n$, the following are equivalent:
\begin{enumerate}
\item The multiset $$\underbrace{q^2a,\ldots,q^2a}_{d_1\mbox{\small-times}},\underbrace{q^4a,\ldots,q^4a}_{d_2\mbox{\small-times}},\ldots,\underbrace{q^{2n}a,\ldots,q^{2n}a}_{d_n\mbox{ \small-times}}$$ is the union of $q$-strings $$S_{k_1,q^{2l_1}a},\ldots, S_{k_s,q^{2l_s}a}$$ which are pairwise in general position,
\item The rigid representation for the linearly oriented quiver of type $A_n$ of dimension vector ${\bf d}$ is given as the direct sum
$$U_{l_1,k_1+l_1-1}\oplus\dots\oplus U_{l_s,k_s+l_s-1}.$$
\end{enumerate}
\qed
\end{lem}

The decomposition of the rigid representation of a fixed dimension vector into indecomposables has been described in \cite{KR} for any Dynkin quiver by a piecewise-linear formula. We get the following corollary.

\begin{cor} Given ${\bf d} \in \mathbb{Z}_+^n$, we set
\[
k_{ij} = \max \Big\{ 0, \min_{} \{d_k - d_{i-1},\, d_k - d_{j+1}: i \le k \le j\} \Big\},\ 1\leq i\leq j\leq n
\]
with the convention $d_0 = d_{n+1}= 0$ and define the Drinfeld polynomial $\boldsymbol{\pi}(u)=\prod_{i=1}^n(1-q^{2i}au)^{d_i}$.
\begin{enumerate} 
  \item Consider the multiset
\[
\underbrace{q^2 a, \ldots, q^2 a}_{d_1 \text{ times}}, \;\;
\underbrace{q^4 a, \ldots, q^4 a}_{d_2 \text{ times}}, \;\;
\ldots, \;\;
\underbrace{q^{2n} a, \ldots, q^{2n} a}_{d_n \text{ times}}.
\]
This multiset decomposes into a union of $q$-strings, which are pairwise in general position:
\[
\bigcup_{1 \le i \le j \le n} S_{j-i+1,\, q^{2i} a}^{\, k_{ij}}.
\]
\item We have an isomorphism of $\mathbf{U}_q$-modules 
$$V(\boldsymbol{\pi})\cong \bigotimes_{1\leq i\leq j\leq n}W_{j-i+1,q^{2i}a}^{\otimes k_{ij}}$$
and thus arriving at a piecewise-linear formula for the $q$-character of all simple representations whose Drinfeld polynomials has zeroes in a single $(q\cdot)$-orbit in $\mathbb{C}^*$:
$$\chi_q(V(\boldsymbol{\pi}))=\prod_{1\leq i\leq j\leq n}\left(Y_{q^{2i}a}\cdots Y_{q^{2j}a}\cdot\left(1+\sum_{t=i}^j A_{q^{2t+1}a}^{-1}A_{q^{2t+3}a}^{-1}\cdots A_{q^{2j+1}a}^{-1}\right)\right)^{k_{ij}}.$$
\end{enumerate}
\end{cor}

\begin{proof}
The first part of the corollary is a consequence of \cite[Section 3.1]{KR} and Lemma~\ref{gp} and the second part follows from \cite[Theorem 4.8]{CP91} and part (1).
\end{proof}
The $q$-character of the irreducible representations for quantum affine $\mathfrak{sl}_2$ are well-known. We emphasize that the new aspect is a piecewise-linear formula that determines the tensor product multiplicities of a Kirillov–Reshetikhin module directly from the corresponding Drinfeld polynomial.
\bibliographystyle{plain}
\bibliography{ref}

\end{document}